\documentclass[12pt]{article}

\usepackage{amssymb,amsmath,amsthm}
\usepackage[shortlabels]{enumitem}
\usepackage{tikz}
\usepackage[T1]{fontenc}

\usepackage{xcolor}

\newcommand{\Z}{\mathbb{Z}}
\newcommand{\R}{\mathbb{R}}

\newcommand{\F}{\mathbb{F}}
\newcommand{\X}{\mathbf{X}}
\newcommand{\Y}{\mathbf{Y}}

\newcommand{\cc}{\mathfrak{c}}
\newcommand{\OO}{\mathcal{O}}

\renewcommand{\AA}{\mathfrak{A}}

\newcommand{\valpha}{\vec{\alpha}}
\renewcommand{\d}{\vec{d}}
\newcommand{\vomega}{\vec{\omega}}

\newcommand{\ra}{\rightarrow}

\newcommand{\tst}{\textstyle}

\newcommand{\Gt}{\widetilde{G}}
\newcommand{\id}{\rm{id}}
\newcommand{\oneb}{\overline{1}}
\newcommand{\sigmab}{\overline{\sigma}}
\newcommand{\taub}{\overline{\tau}}
\newcommand{\onet}{\widetilde{1}}
\newcommand{\sigmat}{\widetilde{\sigma}}
\newcommand{\taut}{\widetilde{\tau}}

\DeclareMathOperator{\ch}{char}

\DeclareMathOperator{\Gal}{Gal}
\DeclareMathOperator{\Frac}{Frac}

\DeclareMathOperator{\rank}{rank}

\newcommand{\Lrightarrow}{\hbox to1cm{\rightarrowfill}}
\newcommand{\Ldownarrow}{\bigg\downarrow}

\newtheorem{theorem}{Theorem}
\newtheorem{lemma}[theorem]{Lemma}
\newtheorem{prop}[theorem]{Proposition}
\newtheorem{cor}[theorem]{Corollary}

\theoremstyle{definition}
\newtheorem{remark}[theorem]{Remark}
\newtheorem{definition}[theorem]{Definition}
\newtheorem{example}[theorem]{Example}

\numberwithin{equation}{section}
\numberwithin{theorem}{section}

\title{Galois scaffolds for $p$-extensions in
characteristic $p$}
\author{G. Griffith Elder \\
Department of Mathematics \\
University of Nebraska Omaha \\
Omaha, NE 68182 \\
USA \\[.2cm]
{\tt elder@unomaha.edu}
\and
Kevin Keating \\
Department of Mathematics \\
University of Florida \\
Gainesville, FL 32611 \\
USA \\[.2cm]
{\tt keating@ufl.edu}}
\begin{document}

\maketitle

\begin{abstract}
Let $K$ be a local field of characteristic $p>0$ with
perfect residue field and let $G$ be a finite $p$-group.
In this paper we use Saltman's construction of a generic
$G$-extension of rings of
characteristic $p$ to construct totally ramified
$G$-extensions $L/K$ that have Galois scaffolds.  We
specialize this construction to produce $G$-extensions
$L/K$ such that the ring of integers $\OO_L$ is free of
rank 1 over its associated order $\AA_0$, and extensions
such that $\AA_0$ is a Hopf order in the group ring
$K[G]$.
\end{abstract}
\\[2mm]

\section{Introduction}

Let $p$ be prime and let $G$ be a group of order $p^n$.
In \cite{salt} Saltman constructed a Galois ring
extension $S/R$ with Galois group $G$, where $S$ and
$R$ are polynomial rings in $n$ variables over
$\F_p=\Z/p\Z$.  Saltman's extension is generic in the
sense that every $G$-extension of commutative rings of
characteristic $p$ is induced by $S/R$.  In this paper
we use a slightly modified version of Saltman's
construction to answer some existence questions
regarding $G$-extensions of local fields of
characteristic $p$.

     Let $K$ be a local field of characteristic $p$ and
let $u_1<u_2<\cdots<u_n$ be positive integers which are
relatively prime to $p$.  Maus \cite{maus} showed that
if $u_i>pu_{i-1}$ for $2\le i\le n$ then there is a
totally ramified $C_{p^n}$-extension $L/K$ whose upper
ramification breaks are $u_1,u_2,\ldots,u_n$.  We use
generic extensions to generalize Maus's result: Given a
$p$-group $G$ and a composition series for $G$, there
exists a constant $M\ge1$ that depends only on $G$ and
the composition series, such that if $u_i>Mu_{i-1}$ for
$2\le i\le n$ then there is a totally ramified
$G$-extension $L/K$ whose upper ramification breaks are
$u_1,u_2,\ldots,u_n$.
     
     Let $K$ be a local field with residue
characteristic $p$ and let $L/K$ be a finite totally
ramified Galois extension whose Galois group
$G=\Gal(L/K)$ is a $p$-group.  A Galois scaffold for
$L/K$ is a set of data that facilitates computation
of the Galois module structure of the ring of integers
$\OO_L$ of $L$ and of its ideals \cite{bce}.  While it
seems clear that for most extensions a Galois scaffold
cannot be constructed, many of the wildly ramified
Galois $p$-extensions $L/K$ for which there is some
understanding of the Galois module structure of $\OO_L$
do in fact admit a Galois scaffold.  In this paper we
show that if $\ch(K)=p$ then for every $p$-group $G$
there exist $G$-extensions with Galois scaffolds.  As
applications we show that for every $p$-group $G$ there
are $G$-extensions $L/K$ such that the ring of integers
$\OO_L$ of $L$ is free of rank 1 over its associated
order $\AA_0$, and there are $G$-extensions such
that $\AA_0$ is a Hopf order.  Hence our constructions
produce an interesting new family of Hopf orders in the
group ring $K[G]$.

     Throughout the paper we let $K$ be a local field
with perfect residue field; unless otherwise stated, $K$
has characteristic $p$.  Let $K^{sep}$ be a separable
closure of $K$.  For each finite subextension $L/K$ of
$K^{sep}/K$ let $v_L$ be the valuation on $K^{sep}$
normalized so that $v_L(L^{\times})=\Z$ and let $\OO_L$
be the ring of integers of $L$.

\section{$p$-filtered groups}

In this section we give the definition of $p$-filtered
groups and record some basic facts about these objects.

\begin{definition}
A {\em $p$-filtered group} is a pair $(G,\{G_{(i)}\})$
consisting of a group $G$ of order $p^n$ and a
composition series
\[\{1\}=G_{(n)}< G_{(n-1)}<\dots<G_{(1)}<G_{(0)}=G\]
for $G$ such that $G_{(i)}\trianglelefteq G$ and
$|G_{(i)}|=p^{n-i}$ for $0\le i\le n$.
\end{definition}

     We often denote the $p$-filtered group
$(G,\{G_{(i)}\})$ simply by $G$.  If $G$ is a
$p$-filtered group then $G/G_{(i)}$ is also a $p$-filtered
group, with subgroups
\[G_{(i)}/G_{(i)}<G_{(i-1)}/G_{(i)}<\cdots<G_{(1)}/G_{(i)}
< G_{(0)}/G_{(i)}.\]
Let $G$ be a $p$-filtered group of order $p^n$.  Define
\[\Sigma_G=\{1\le i\le n:\text{The extension $G/G_{(i)}$
of $G/G_{(i-1)}$ by $G_{(i-1)}/G_{(i)}$ is split.}\}.\]
In addition, for $0\le i\le n$ set
$\Sigma_G^i=\{j\in \Sigma_G:j\le i\}$.

     For a finite group $G$ we let $\Phi(G)$ denote the
Frattini subgroup of $G$.  Thus $\Phi(G)$ is the
intersection of the maximal proper subgroups of $G$.
Let $G$ and $H$ be finite groups.  We note the following
facts, which may be found in \cite{mil}:
\begin{enumerate}
\item $\Phi(G\times H)=\Phi(G)\times\Phi(H)$.  
\item If $G$ is a $p$-group then $\Phi(G)$ is the
smallest $N\trianglelefteq G$ such that $G/N$ is an
elementary abelian $p$-group (Burnside's basis
theorem).
\end{enumerate}
The rank of the $p$-group $G$ is defined to be the rank
of the elementary abelian $p$-group $G/\Phi(G)$.  It
follows that $\rank(G)$ is equal to the cardinality of
any minimal generating set for $G$.  We will need the
following elementary result:

\begin{prop}
For $1\le i\le n$ we have
\[\rank(G/G_{(i)})=\begin{cases}
\rank(G/G_{(i-1)})+1&\text{if }i\in\Sigma_G, \\
\rank(G/G_{(i-1)})&\text{if }i\not\in\Sigma_G.
\end{cases}\]
\end{prop}

\begin{proof}
If $i\in \Sigma_G$ then since $G/G_{(i)}$ is a $p$-group
and $G_{(i-1)}/G_{(i)}$ is cyclic of order $p$ we have
$G/G_{(i)}\cong(G/G_{(i-1)})\times C_p$, and hence
\[(G/G_{(i)})/\Phi(G/G_{(i)})
\cong((G/G_{(i-1)})/\Phi(G/G_{(i-1)}))\times C_p.\]
Therefore $\rank(G/G_{(i)})=\rank(G/G_{(i-1)})+1$.  If
$i\not\in\Sigma_G$ let $A$ be a subset of $G$ such that
$|A|=\rank(G/G_{(i-1)})$ and $\{aG_{(i-1)}:a\in A\}$
generates $G/G_{(i-1)}$.  Then
$H=\langle aG_{(i)}:a\in A\rangle$ is a subgroup of
$G/G_{(i)}$ such that $|H|\ge|G/G_{(i-1)}|=p^{i-1}$.  If
$|H|=p^{i-1}$ then $H\cap(G_{(i-1)}/G_{(i)})$ is
trivial.  Hence $G/G_{(i)}$ is the product of $H$ and
the central subgroup $G_{(i-1)}/G_{(i)}$, which
contradicts the assumption $i\not\in\Sigma_G$.
Therefore $H=G/G_{(i)}$, and hence
$\rank(G/G_{(i)})=\rank(G/G_{(i-1)})$.
\end{proof}

\begin{cor} \label{easy}
For $1\le i\le n$ we have $\rank(G/G_{(i)})=|\Sigma_G^i|$.
In particular, $\rank(G)=|\Sigma_G|$.
\end{cor}

\section{Generic $G$-extensions of commutative rings}
\label{generic}

Let $G$ be a $p$-group.
In this section we describe a version of Saltman's
construction of a generic $G$-extension of commutative
rings \cite{salt}.  The generic $G$-extension $S/R$
constructed here is somewhat more general than that
given in \cite{salt}, in that we don't require the
Frattini subgroup of the $p$-group $G$ to appear in our
filtration of $G$.  Unlike Saltman, who considers
specializations of $S/R$ to ring extensions, we
only consider field extensions, since that is the case
that we need for our applications.

\begin{definition}
Let $S$ be a commutative ring with 1, let $G$ be a
finite group of automorphisms of $S$, and set
\[R=S^G=\{x\in S:\sigma(x)=x\text{ for all }\sigma\in G\}.\]
Say that $S/R$ is a Galois extension with group $G$ if
for every maximal ideal $M\subset S$ and every
$\sigma\in G$ with $\sigma\not=1$ there is $s\in S$ with
$\sigma(s)-s\not\in M$.
\end{definition}

     See \cite[p.\,81]{DI} for alternative
characterizations of Galois extensions of rings.  In
general, for a ring extension $S/R$ there may exist more
than one group $G$ of automorphisms of $S$ such that
$R=S^G$ and $S/R$ is Galois with group $G$.  However, if
$S/R$ is a Galois extension with group $G$ and $S,R$ are
integral domains then by setting $E=\Frac(S)$ and
$F=\Frac(R)$ we get a Galois extension of fields $E/F$
such that $\Gal(E/F)\cong G$.  In this case $G$ is equal
to the group of all $R$-automorphisms of $S$, so it
makes sense to say that $S/R$ is a Galois extension
without specifying a group of automorphisms of $S$.

     If $R$ is a ring of characteristic $p$ then the
simplest nontrivial Galois $p$-extensions of $R$ are
Artin-Schreier extensions.  Saltman gives some
properties of these extensions in Theorem~1.3 of
\cite{salt}:

\begin{prop}
Let $R$ be a ring of characteristic $p$, let $c\in R$,
and set $S=R[X]/(X^p-X-c)$.  Set $v=X+(X^p-X-c)$ and let
$\sigma$ be the unique automorphism of $S$ which fixes
$R$ and satisfies $\sigma(v)=v+1$.  Then $S/R$ is a
Galois extension with group $\langle\sigma\rangle$.
\end{prop}

     We will also use the following fact, which is
proved as Corollary~1.3(3) in Chapter~III of \cite{DI}:

\begin{prop} \label{tensor}
Let $S/R$ be a Galois extension of rings with group
$G$ and let $T$ be a commutative $R$-algebra.  Then the
action of $G$ on $T\otimes_RS$ defined by
$\sigma(t\otimes s)=t\otimes\sigma(s)$ makes
$T\otimes_RS$ a Galois extension of $T$.
\end{prop}

     Let $S$ be a ring of characteristic $p$, and let
$G$ be a group of automorphisms of $S$ such that
$|G|=p^n$ and $S$ is a Galois extension of the subring
$R=S^G$ fixed by $G$.  In Lemma~1.1 of \cite{salt} it is
observed that $H^q(G,S)=0$ for all $q\ge1$.  Let $\Gt$
be a group of order $p^{n+1}$, let $\pi:\Gt\ra G$ be an
onto homomorphism, and set $H=\ker(\pi)$.  Let
$u:G\ra\Gt$ be a section of $\pi$.  Then the map
$g:G\times G\ra H$ defined by
$g(\sigma,\tau)=u(\sigma)u(\tau)u(\sigma\tau)^{-1}$ is a
2-cocycle.  Let $\chi:H\ra\F_p$ be an isomorphism; then
$c(\sigma,\tau)=\chi(g(\sigma,\tau))$ is a 2-cocycle
with values in $\F_p\subset S$.  Since $H^2(G,S)=0$
there is a cochain $(s_{\sigma})_{\sigma\in G}$ with
values in $S$ such that $c(\sigma,\tau)
=s_{\sigma}+\sigma(s_{\tau})-s_{\sigma\tau}$ for all
$\sigma,\tau\in G$.  Let $\wp(X)=X^p-X\in\F_p[X]$ be the
Artin-Schreier polynomial.  Since
$c(\sigma,\tau)\in\F_p$ we have
$\wp(s_{\sigma})+\sigma(\wp(s_{\tau}))
=\wp(s_{\sigma\tau})$ for all $\sigma,\tau\in G$.  Thus
$(\wp(s_{\sigma}))_{\sigma\in G}$ is a 1-cocycle with
values in $S$.  Since $H^1(G,S)=0$ there is $d\in S$
such that $\wp(s_{\sigma})=\sigma(d)-d$ for all
$\sigma\in G$.

     In Lemma~1.8 of \cite{salt}, Saltman proved the
following facts:

\begin{lemma} \label{basic}
Let $S/R$ be a Galois extension with group $G$, and let
$\Gt$, $H$, $d$ be as above.
\begin{enumerate}[(a)]
\item View $T=S[X]/(X^p-X-d)$ as an extension of $S$.
The group $G$ of automorphisms of $S$ extends to a group
of automorphisms of $T$ which is isomorphic to $\Gt$ and
makes $T/R$ a Galois extension.
\item Suppose $T'$ is an extension of $S$ such that
$T'/R$ is Galois with group $\Gt$ and $S$ is the fixed
ring of $H$. Then for some $r\in R$ there is an
isomorphism of $S$-algebras $T'\cong S[X]/(X^p-X-d-r)$.
\item If $S$ has no nontrivial idempotents and the
extension $\Gt$ of $G$ by $H$ is not split then
$d\not\in\wp(S)+R$.
\end{enumerate}
\end{lemma}

     Using this lemma, we construct the generic
$G$-extension:

\begin{prop} \label{Di}
Let $(G,\{G_{(i)}\})$ be a $p$-filtered group
of order $p^n$.  Then for $1\le i\le n$ there are
polynomials $D_i\in\F_p[Y_1,\ldots,Y_{i-1}]$ with the
following properties:
\begin{enumerate}[(a)]
\item $D_i=0$ for $i\in \Sigma_G$, and $D_i\not\in\F_p$
for $i\not\in \Sigma_G$.
\item For $0\le i\le n$ set $R_i=\F_p[X_1,\ldots,X_i]$,
and define $S_0,S_1,\ldots,S_n$ recursively by
$S_0=\F_p$ and
$S_i=S_{i-1}[Y_i,X_i]/(Y_i^p-Y_i-D_i-X_i)$ for
$1\le i\le n$.  Then $S_i\cong\F_p[Y_1,\ldots,Y_i]$ and
$S_i/R_i$ is a Galois extension.
\item For $1\le i\le n$ let
$\pi_i:\Gal(S_i/R_i)\ra\Gal(S_{i-1}/R_{i-1})$ be the
homomorphism induced by restriction.  Then there are
isomorphisms $\lambda_i:\Gal(S_i/R_i)\ra G/G_{(i)}$ such
that for $1\le i\le n$ the following diagram commutes:
\[\setlength{\arraycolsep}{1pt}
\begin{array}{*{9}c}
\Gal(S_i/R_i)&\overset{\textstyle\pi_i}{\Lrightarrow}
&\Gal(S_{i-1}/R_{i-1}) \\[1mm]
\lambda_i\Ldownarrow
&&\hspace*{-8mm}\lambda_{i-1}\Ldownarrow&& \\[4mm]
G/G_{(i)}&\Lrightarrow&G/G_{(i-1)}.
\end{array}\]
\end{enumerate}
\end{prop}

\begin{proof}
Let $1\le i\le n$ and assume that for $1\le j<i$ we have
constructed $D_j$, $S_j$, $\lambda_j$ satisfying the
conditions of the proposition.  By Lemma~\ref{basic}(a)
there is $D_i\in S_{i-1}$ such that
$S_{i-1}[Y_i]/(Y_i^p-Y_i-D_i)$ is Galois over $R_{i-1}$,
with Galois group $G/G_{(i)}$.  If $i\in\Sigma_G$ then
the extension $G/G_{(i)}$ of $G/G_{(i-1)}$ by
$G_{(i-1)}/G_{(i)}$ is split, so we may assume $D_i=0$.
On the other hand, if $i\not\in\Sigma_G$ then by
Lemma~\ref{basic}(c) we get $D_i\not\in\F_p$,
and by Lemma~\ref{basic}(a)
$S_{i-1}[Y_i]/(Y_i^p-Y_i-D_i)$ is Galois over $R_{i-1}$,
with Galois group $G/G_{(i)}$.  It then follows from
Proposition~\ref{tensor} that
$S_{i-1}[Y_i,X_i]/(Y_i^p-Y_i-D_i)$ is Galois over
$R_i=R_{i-1}[X_i]$, again with Galois group $G/G_{(i)}$.
Since $(\sigma-1)(D_i+X_i)=(\sigma-1)(D_i)$ for all
\[\sigma\in\Gal(S_{i-1}[X_i]/R_i)\cong
\Gal(S_{i-1}/R_{i-1}),\]
it follows from Lemma~\ref{basic}(a) that
\[S_i=S_{i-1}[Y_i,X_i]/(Y_i^p-Y_i-D_i-X_i)\]
is also Galois over $R_i$, with
$\Gal(S_i/R_i)\cong G/G_{(i)}$.
\end{proof}

     We now show that $S_n/R_n$ is a generic
$G$-extension, in the sense that if $F$ is a field of
characteristic $p$ such that $F/\wp(F)$ is sufficiently
large, then all $G$-extensions $E/F$ are
specializations of $S_n/R_n$.

\begin{theorem} \label{main}
Let $(G,\{G_{(i)}\})$ be a $p$-filtered group of order
$p^n$ and set $r=\rank(G)$.  For $1\le i\le n$ let
$D_i\in\F_p[Y_1,\ldots,Y_{i-1}]$ be polynomials
satisfying the conditions of Proposition~\ref{Di}.  Let
$F$ be a field of characteristic $p$ such that
$\dim_{\F_p}(F/\wp(F))\ge r$.
\begin{enumerate}[(a)]
\item Let $a_1,\ldots,a_n$ be elements of $F$ such that
$\{a_j+\wp(F):j\in\Sigma_G\}$ is an $\F_p$-linearly
independent subset of $F/\wp(F)$.  Define
$F_0,F_1,\ldots,F_n$ recursively by $F_0=F$ and
$F_i=F_{i-1}(\alpha_i)$ for $1\le i\le n$, where
$\alpha_i\in F^{sep}$ satisfies
$\alpha_i^p-\alpha_i=d_i+a_i$ with
\[d_i=D_i(\alpha_1,\ldots,\alpha_{i-1}).\]
Then for $0\le i\le n$, $F_i/F$ is a Galois field
extension and there is an isomorphism
$\mu_i:\Gal(F_i/F)\ra G/G_{(i)}$.  Furthermore, the
isomorphisms $\mu_i$ may be chosen so that for
$1\le i\le n$ the following diagram commutes:
\begin{equation} \label{diagram}
\setlength{\arraycolsep}{1pt}
\begin{array}{*{9}c}
\Gal(F_i/F)&\Lrightarrow&\Gal(F_{i-1}/F) \\[1mm]
\hspace*{-5mm}\mu_i\Ldownarrow
&&\hspace*{-8mm}\mu_{i-1}\Ldownarrow&& \\[4mm]
G/G_{(i)}&\Lrightarrow&G/G_{(i-1)}.
\end{array}
\end{equation}
\item Conversely, let
$F=F_0\subset F_1\subset\cdots\subset F_n$ be a tower of
Galois field extensions of $F$ such that there is an
isomorphism $\mu:\Gal(F_n/F)\ra G$ with
$\mu(\Gal(F_n/F_i))=G_{(i)}$ for $0\le i\le n$.  Then
there are $a_j\in F$ such that
$F_i=F(\alpha_1,\ldots,\alpha_i)$ for $0\le i\le n$,
where $\alpha_j$ are defined in terms of $a_j$ as in
(a).
\end{enumerate}
\end{theorem}

\begin{proof}
(a) For $0\le i\le n$ define a ring
homomorphism $\psi_i:S_i\ra F_i$ by
$\psi_i(Y_j)=\alpha_j$ for $1\le j\le i$.  Then
$\psi_i(X_j)=a_i$ for $1\le j\le i$, so
$\psi_i(R_i)\subset F$.  Viewing $S_{i-1}$ as a subring
of $S_i$ we get the compatibility conditions
$\psi_i|_{S_{i-1}}=\psi_{i-1}$ for $1\le i\le n$.  We
use induction on $i$.  The base case $i=0$ is trivial.
Let $1\le i\le n$ and assume that the statement holds
for $i-1$.  We claim that $d_i+a_i\not\in\wp(F_{i-1})$.
By the inductive hypothesis $F_{i-1}/F$ is Galois, with
$\Gal(F_{i-1}/F)\cong G/G_{(i-1)}$.  If
$i\not\in\Sigma_G$ then $G/G_{(i)}$ is a nonsplit
extension of $G/G_{(i-1)}$ by $G_{(i-1)}/G_{(i)}$, so
the claim follows from Lemma~\ref{basic}(c).  Suppose
$i\in \Sigma_G$.  By Corollary~\ref{easy} we have
$\rank(\Gal(F_{i-1}/F))=|\Sigma_G^{i-1}|$.  Since
$\{a_j+\wp(F):j\in\Sigma_G^i\}$ is an $\F_p$-linearly
independent subset of $F/\wp(F)$, $F_i/F$ contains an
elementary abelian subextension of rank
$|\Sigma_G^i|=|\Sigma_G^{i-1}|+1$.  Hence
\[\rank(\Gal(F_i/F))>\rank(\Gal(F_{i-1}/F)).\]
It follows that $F_i\not=F_{i-1}$, so
$d_i+a_i=a_i\not\in\wp(F_{i-1})$.  In both cases we get
$[F_i:F_{i-1}]=p$, and hence $[F_i:F]=p^i$.  The map
$\psi_i:S_i\ra F_i$ induces an onto homomorphism
$F\otimes_{R_i}S_i\ra F_i$.  Since $S_i$ is a free
$R_i$-module of rank $p^i$, this map is an isomorphism.
Hence by Proposition~\ref{tensor} we see that $F_i/F$ is
a Galois extension, with
$\Gal(F_i/F)\cong\Gal(S_i/R_i)$.  Therefore by
Proposition~\ref{Di}(c) there is an isomorphism
$\mu_i:\Gal(F_i/F)\ra G/G_{(i)}$ which makes the diagram
(\ref{diagram}) commute.
\\[\smallskipamount]
(b) We use induction on $i$.  Note that for
$0\le i\le n$, $\mu$ induces an isomorphism
$\mu_i:\Gal(F_i/F)\ra G/G_{(i)}$.  Suppose we have
$a_1,\ldots,a_{i-1}\in F$ such that
$F_{i-1}=F(\alpha_1,\ldots,\alpha_{i-1})$.  Set
$d_i=D_i(\alpha_1,\ldots,\alpha_{i-1})$.  If
$i\in\Sigma_G$ then
$G/G_{(i)}\cong(G/G_{(i-1)})\times C_p$ and $d_i=0$.
Hence there is $a_i\in F$ such that
$F_i=F_{i-1}(\alpha_i)$, with
$\alpha_i^p-\alpha_i=a_i=d_i+a_i$.  Suppose
$i\not\in\Sigma_G$.  Then by Lemma~\ref{basic}(b) there
is $a_i\in F$ such that
$F_i\cong F_{i-1}[Y]/(Y^p-Y-d_i-a_i)$.  Hence
$F_i=F(\alpha_1,\ldots,\alpha_{i-1},\alpha_i)$, with
$\alpha_i$ a root of $Y^p-Y-d_i-a_i$.
\end{proof}

\begin{remark}
Saltman \cite[pg.~308]{salt} states that his results
``can be viewed as a generalization of the theory of
Witt vectors''.  In particular, he proves the existence
of polynomials $D_i$ which satisfy the conditions of
Proposition~\ref{Di} and Theorem~\ref{main}.  These
polynomials depend only on the $p$-filtered group $G$,
and not on the base field $F$.  In the case where $G$ is
a cyclic $p$-group one can use Witt addition polynomials
to produce $D_i$ satisfying Saltman's conditions.
\end{remark}

\section{Ramification breaks in $G$-extensions}

Let $K$ be a local field of characteristic $p$ with
perfect residue field and let $(G,\{G_{(i)}\})$ be a
$p$-filtered group of order $p^n$.  Let
$u_1<u_2<\cdots<u_n$ be positive integers such that
$p\nmid u_i$ for $1\le i\le n$.  We wish to show that if
this sequence grows quickly enough then there is a
totally ramified $G$-extension $L/K$ such that every
ramification subgroup of $\Gal(L/K)$ is equal to
$G_{(i)}$ for some $i$ and $u_1,u_2,\ldots,u_n$ are the
upper ramification breaks of $L/K$.

     We begin by recalling some basic facts about higher
ramification theory; see Chapter~IV of \cite{cl} for
more information on this topic.  Let $K$ be a local
field and let $L/K$ be a Galois extension.  Set
$G=\Gal(L/K)$ and let $G_0$ be the inertia subgroup of
$G$.  Let $\pi_L$ be a uniformizer for $L$.
We define the ramification number of $\sigma\in G$
to be $i(\sigma)=v_L(\sigma(\pi_L)-\pi_L)-1$ if
$\sigma\in G_0$, and $i(\sigma)=-1$ if
$\sigma\not\in G_0$.  (Beware that $i_G(\sigma)$ from
\cite{cl} is is not the same as $i(\sigma)$: instead we
have $i_G(\sigma)=i(\sigma)+1$.) Then
$i(\id_L)=+\infty$, and $i(\sigma)$ is a nonnegative
integer for $\sigma\in G_0\smallsetminus\{\id_L\}$.  For
$t\in\R$ with $t\ge-1$ define the $t$th lower
ramification subgroup of $G$ to be
$G_t=\{\sigma\in G:i(\sigma)\ge t\}$.  Say $b\ge-1$ is a
lower ramification break of $L/K$ if
$G_b\not=G_{b+\epsilon}$ for all real $\epsilon>0$.
Thus $b$ is a lower ramification break of $L/K$ if and
only if $b=i(\sigma)$ for some $\sigma\in G$ with
$\sigma\not=\id_L$.

     We define the Hasse-Herbrand function
$\phi_{L/K}:[-1,\infty)\ra[-1,\infty)$ by
\[\phi_{L/K}(x)=\int_0^x\frac{dt}{|G_0:G_t|}.\]
Then $\phi_{L/K}$ is continuous on $[-1,\infty)$ and
differentiable on $(-1,\infty)$ except at the lower
ramification breaks.  Since $\phi_{L/K}$ is one-to-one
and onto it has an inverse
$\psi_{L/K}:[-1,\infty)\ra[-1,\infty)$.  Define the
upper ramification subgroups of $G$ by setting
$G^x=G_{\psi_{L/K}(x)}$ for $x\ge-1$.  Say that $u\ge-1$
is an upper ramification break of $L/K$ if
$G^u\not=G^{u+\epsilon}$ for all $\epsilon>0$.  Then
$\psi_{L/K}$ is differentiable except at the upper
ramification breaks of $L/K$, and $u$ is an upper
ramification break of $L/K$ if and only if
$\psi_{L/K}(u)$ is a lower ramification break.  Let
$M/K$ be a Galois subextension of $L/K$ and set
$H=\Gal(L/M)$.  Then by Herbrand's theorem
\cite[IV\,\S3]{cl} we get
$\phi_{L/K}=\phi_{M/K}\circ\phi_{L/M}$ and
$\psi_{L/K}=\psi_{L/M}\circ\psi_{M/K}$.  Furthermore, we
have $(G/H)^x=G^xH/H$ for all $x\ge-1$.  It follows that
if $u$ is an upper break of $M/K$ then $u$ is also an
upper break of $L/K$.

\begin{lemma} \label{breakshift}
Let $L/K$ be a finite Galois extension and let $E/K$ be
a ramified $C_p$-extension such that $E\not\subset L$.
Assume that the unique (upper and lower) ramification
break $v$ of $E/K$ is not an upper ramification break of
$L/K$.  Then the ramification break of the
$C_p$-extension $LE/L$ is $\psi_{L/K}(v)$.
\end{lemma}

\begin{proof}
Since $\psi_{LE/E}\circ\psi_{E/K}
=\psi_{LE/L}\circ\psi_{L/K}$ is not differentiable at
$v$, but $\psi_{L/K}$ is differentiable at $v$, we
deduce that $\psi_{LE/L}$ is not differentiable at
$\psi_{L/K}(v)$.  Hence $\psi_{L/K}(v)$ is the unique
upper break of $LE/L$.
\end{proof}

     We will mainly be considering totally ramified
Galois extensions $L/K$ degree $p^n$ with the property
that for every lower ramification break $b$ we have
$|G_b:G_{b+\epsilon}|=p$.  In this case there are $n$
lower breaks $b_1<b_2<\dots<b_n$ and $n$ upper breaks
$u_1<u_2<\dots<u_n$.  The breaks are related by the
formulas $u_1=b_1$ and $u_{i+1}-u_i=p^{-i}(b_{i+1}-b_i)$
for $1\le i\le n-1$.  As a result we get the following
inequalities:

\begin{lemma} \label{known}
Let $1\le i\le j\le n$.  Then
\begin{enumerate}[(a)]
\item $b_j-b_i\le p^{j-1}(u_j-u_i)$.
\item $b_j\le p^{j-1}u_j<p^ju_j$.
\end{enumerate}
\end{lemma}

\begin{proof}
(a) If $i=j$ the claim is clear.  If $i<j$ then
\begin{align*}
b_j-b_i&=\sum_{h=i}^{j-1}\,(b_{h+1}-b_h) \\
&=\sum_{h=i}^{j-1}p^h(u_{h+1}-u_h) \\
&=p^{j-1}u_j-p^iu_i
+\sum_{h=i+1}^{j-1}(p^{h-1}-p^h)u_h \\
&\le p^{j-1}u_j-p^iu_i
+\sum_{h=i+1}^{j-1}(p^{h-1}-p^h)u_i \\
&=p^{j-1}(u_j-u_i).
\end{align*}
(b) This follows from (a) by letting $i=1$.
\end{proof}

     The following well-known fact will often be used
without comment (cf.\ Proposition~2.5 in
\cite[III]{FV}).

\begin{lemma} \label{ASram}
Let $K$ be a local field of characteristic $p$ and let
$L/K$ be a ramified $C_p$-extension.  Let $s\in K$ be
such that $L$ is generated over $K$ by a root $\alpha$
of $X^p-X-s$.  Then the following hold:
\begin{enumerate}[(a)]
\item The ramification break $b$ of $L/K$ satisfies
$b\le-v_K(s)$, with equality if $p\nmid v_K(s)$.
\item If $b<-v_K(s)$ then there is $t\in K$ such that
that $v_K(s-\wp(t))=-b$.
\end{enumerate}
\end{lemma}

     Let $G$ be a $p$-filtered group of order $p^n$, let
$a_1,\dots,a_n$ be elements of $K$ which satisfy the
hypotheses of Theorem~\ref{main}(a), and let $K_n/K$ be
the associated $G$-extension.  In some cases we can
compute the ramification data of $K_n/K$ in terms of the
valuations $v_K(a_1),\dots,v_K(a_n)$:

\begin{theorem} \label{ramfilt}
Let $(G,\{G_{(i)}\})$ be a $p$-filtered group
of order $p^n$, and for $1\le i\le n$ let
$D_i\in\F_p[Y_1,\ldots,Y_{i-1}]$ be the polynomials
constructed in Proposition~\ref{Di}.  Let $K$ be a local
field of characteristic $p$ with perfect residue
field.  Let $u_1<u_2<\cdots<u_n$ be positive integers
such that $p\nmid u_i$, and let $a_1,\ldots,a_n$ be
elements of $K$ such that $v_K(a_i)=-u_i$ for
$1\le i\le n$.  As in Theorem~\ref{main} we define
$K_0,K_1,\ldots,K_n$ recursively by $K_0=K$ and
$K_i=K_{i-1}(\alpha_i)$ for $1\le i\le n$, where
$\alpha_i$ satisfies $\alpha_i^p-\alpha_i=d_i+a_i$ with
$d_i=D_i(\alpha_1,\ldots,\alpha_{i-1})$.  Define
$b_1<b_2<\cdots<b_n$ recursively by $b_1=u_1$ and
$b_{i+1}-b_i=p^i(u_{i+1}-u_i)$ for $1\le i\le n-1$.  If
$b_i>-p^{i-1}v_K(d_i)$ for all $i\not\in\Sigma_G$ then
\begin{enumerate}[(a)]
\item $K_n/K$ is Galois, and there is an isomorphism
$\mu:\Gal(K_n/K)\ra G$ such that
$\mu(\Gal(K_n/K_i))=G_{(i)}$ for $0\le i\le n$.
\item $K_n/K$ has upper ramification breaks
$u_1,u_2,\ldots,u_n$ and lower ramification breaks
$b_1,b_2,\ldots,b_n$.  Furthermore, we have
$v_K(\alpha_i)=-p^{-1}u_i$ for $0\le i\le n$.
\item The ramification subgroups of $\Gal(K_n/K)$ are
the subgroups of the form $\Gal(K_n/K_i)$ for
$0\le i\le n$.
\end{enumerate}
\end{theorem}

\begin{proof}
Since the elements of $\{u_i:i\in\Sigma_G\}$ are
distinct and relatively prime to $p$, the set
$\{a_i+\wp(K):i\in\Sigma_G\}$ is linearly independent
over $\F_p$.  Therefore by Theorem~\ref{main}(a) the
extensions $K_0\subset K_1\subset\cdots\subset K_n$
associated to $a_1,\ldots,a_n$ satisfy condition (a).
We use induction on $i$ to show that
$v_K(\alpha_i)=-p^{-1}u_i$ and $K_i/K$ has upper
ramification breaks $u_1,\dots,u_i$.  It then follows
that $\lvert G^{u_i}\rvert =p^{n-i+1}$, $\lvert
G^{u_i+\epsilon}\rvert=p^{n-i}$, and $K_i/K$ has lower
ramification breaks $b_1,\dots,b_i$.  In addition, since
$\Gal(K_i/K)\cong G/G_{(i)}$ we get
\[G_{(i)}/G_{(i)}=(G/G_{(i)})^{u_i+\epsilon}
=G^{u_i+\epsilon}G_{(i)}/G_{(i)},\]
and hence $G^{u_i+\epsilon}\le G_{(i)}$.  Since
$|G^{u_i+\epsilon}|=|G_{(i)}|=p^{n-i}$ it follows that
$G_{(i)}=G^{u_i+\epsilon}$ is a ramification subgroup of
$G\cong\Gal(K_n/K)$ for $1\le i\le n$.  Therefore the
$n+1$ distinct ramification subgroups of $\Gal(K_n/K)$
are precisely the subgroups $\Gal(K_n/K_i)$ for
$0\le i\le n$.

     We have $D_1=0$, so the upper ramification break of
$K_1/K$ is $-v_K(a_1)=u_1$, and
$v_K(\alpha_1)=p^{-1}v_K(a_1)=-p^{-1}u_1$.  Let
$2\le i\le n$ and assume the claim holds for $i-1$.  If
$i\in\Sigma_G$ then $D_i=0$ and $K(\alpha_i)/K$ is a
$C_p$-extension with upper ramification break $u_i$.
Since $K_i\supset K(\alpha_i)$ it follows that $u_i$ is
an upper ramification break of $K_i/K$.  Hence by
induction $K_i/K$ has upper ramification breaks
$u_1,\ldots,u_{i-1},u_i$.  We also get
$v_K(\alpha_i)=p^{-1}v_K(a_i)=-p^{-1}u_i$.

     Suppose $i\not\in\Sigma_G$, and set
$d_i=D_i(\alpha_1,\dots,\alpha_{i-1})$ as in
Theorem~\ref{main}(a).  By the lower bound on $b_i$ and
Lemma~\ref{known}(b) we get
\[v_K(d_i)>-p^{1-i}b_i\ge-u_i=v_K(a_i).\]
It follows that $v_K(d_i+a_i)=v_K(a_i)=-u_i$, and hence
that $v_K(\alpha_i)=-p^{-1}u_i$.  Since
$\wp(\alpha_i)=d_i+a_i$ we can write
$\alpha_i=\alpha_i'+\alpha_i''$, with
$\wp(\alpha_i')=d_i$ and $\wp(\alpha_i'')=a_i$.  Let
$K_i'=K_{i-1}(\alpha_i')$ and
$K_i''=K_{i-1}(\alpha_i'')$.  We wish to determine the
ramification breaks for the $C_p$-extensions
$K_i'/K_{i-1}$ and $K_i''/K_{i-1}$.

     First consider $K_i'/K_i$. By Theorem~\ref{main}(a)
we have $[K_i':K]=p^i$. Thus $K_i'\not=K_{i-1}$ and
$K_i'/K_{i-1}$ is indeed a $C_p$-extension.  Let $b_i'$
be the ramification break of $K_i'/K_{i-1}$.  Then by
Lemma~\ref{ASram}(a) and the lower bound on $b_i$ we get
$b_i'\le-p^{i-1}v_K(d_i)<b_i$.  By Lemma~\ref{ASram}(b)
there is $\ell'\in K_{i-1}$ such that
$v_{K_{i-1}}(d_i-\wp(\ell'))=-b_i'>-b_i$.  Now consider
$K_i''/K_i$.  Since $a_i\in K$, $K(\alpha_i'')/K$ is a
$C_p$-extension with ramification break $-v_K(a_i)=u_i$.
By Lemma~\ref{breakshift} the ramification break of
$K_i''/K_{i-1}$ is $\psi_{K_{i-1}/K}(u_i)$.  Since
$\psi_{K_i/K}=\psi_{K_i/K_{i-1}}\circ\psi_{K_{i-1}/K}$
and $\psi_{K_i/K_{i-1}}(x)=x$ for $x\leq b_i$, we get
$\psi_{K_{i-1}/K}(u_i)=\psi_{K_i/K}(u_i)=b_i$. Since
$K_i''/K_{i-1}$ has ramification break $b_i$, by
Lemma~\ref{ASram}(b) there is $\ell''\in K_{i-1}$ such
that $v_{K_{i-1}}(a_i-\wp(\ell''))=-b_i$.

     We have shown that $K_i'K_i''/K_{i-1}$ is a
$(C_p\times C_p)$-extension with upper breaks
$b_i'<b_i$.  There are $p+1$ $C_p$-subextensions of
$K_i'K_i''/K_{i-1}$, namely $K_i''$ and
$K_{i-1}(\alpha_i'+s\alpha_i'')$ for $s\in \F_p$.
We are interested in the ramification break for
$K_i/K_{i-1}$, which is the $s=1$ case.  Note that
$K_i=K_{i-1}(\alpha_i-\ell'-\ell'')$, with
\[\wp(\alpha_i-\ell'-\ell'')=(d_i-\wp(\ell'))
+(a_i-\wp(\ell'')).\]
Since $v_{K_{i-1}}(d_i-\wp(\ell'))>
v_{K_{i-1}}(a_i-\wp(\ell''))$ we get
\[v_{K_{i-1}}((d_i-\wp(\ell'))+(a_i-\wp(\ell'')))
=v_{K_{k-1}}(a_i-\wp(\ell''))=-b_i.\]
Since $p\nmid b_i$, it follows that the ramification
break of $K_i/K_{i-1}$ is $b_i$.  Therefore
$b_i$ is a lower ramification break of $K_i/K$, so
$\phi_{K_i/K}(b_i)=u_i$ is an upper ramification break
of $K_i/K$.  Using induction we deduce that $K_i/K$ has
upper ramification breaks $u_1,\ldots,u_{i-1},u_i$.
\end{proof}

     Theorem~\ref{ramfilt} allows us to construct
$G$-extensions which have certain specified sequences of
upper ramification breaks:

\begin{cor} \label{ramfiltcor}
Let $(G,\{G_{(i)}\})$, $K$, $a_i$, $D_i$, $d_i$, $K_i$,
$u_i$, $b_i$ be as in Theorem~\ref{ramfilt}, and for
$i\not\in\Sigma_G$ let $l_i$ denote the total degree of
$D_i$.  If $b_i>p^{i-2}l_iu_{i-1}$ for all
$i\not\in\Sigma_G$ then the conclusions of
Theorem~\ref{ramfilt} hold for $K_0,K_1,\dots,K_n$.
\end{cor}

\begin{proof}
We prove by induction that
$v_K(d_i+a_i)=v_K(a_i)=-u_i$,
$v_K(\alpha_i)=-p^{-1}u_i$, and $b_i>-p^{i-1}v_K(d_i)$
for $1\le i\le n$.  This is clear for
$i\in\Sigma_G$ since $d_i=0$ in this case.  Let
$2\le i\le n$ with $i\not\in\Sigma_G$ and assume the
claim holds for $1\le h<i$.  Then
$v_K(\alpha_h)=-p^{-1}u_h\ge-p^{-1}u_{i-1}$ for
$1\le h<i$.  Using the assumption
$b_i>p^{i-2}l_iu_{i-1}$ we get
$v_K(d_i)\ge-p^{-1}l_iu_{i-1}>-p^{1-i}b_i$, and hence
$b_i>-p^{i-1}v_K(d_i)$.  Lemma~\ref{known}(b) then gives
$v_K(d_i)>-p^{1-i}b_i\ge-u_i=v_K(a_i)$.  It follows that
$v_K(d_i+a_i)=v_K(a_i)$, and hence that
$v_K(\alpha_i)=-p^{-1}u_i$.  Since we have shown that
the hypotheses of Theorem~\ref{ramfilt} hold, the
conclusions of the theorem hold as well.
\end{proof}

     Let $K$ be a local field of characteristic $p$.
Maus \cite{maus} showed that for every sequence of
positive integers $u_1,\ldots,u_n$ such that $p\nmid
u_i$ for $1\le i\le n$ and $u_{i+1}>pu_i$ for $1\le i\le
n-1$ there exists a totally ramified $C_{p^n}$-extension
$L/K$ whose sequence of upper ramification breaks is
$u_1,\ldots,u_n$.  The following corollary shows that a
similar result holds with $C_{p^n}$ replaced by an
arbitrary $p$-filtered group.

\begin{cor} \label{Mbreak}
Let $(G,\{G_{(i)}\})$ be a $p$-filtered group of order
$p^n$.  Then there is $M\ge1$, depending only on
$(G,\{G_{(i)}\})$, with the following property: Let $K$
be a local field of characteristic $p$ and let
$u_1,\ldots,u_n$ be a sequence of positive integers such
that $p\nmid u_i$ for $1\le i\le n$ and $u_{i+1}>Mu_i$
for $1\le i\le n-1$.  Then there exists a totally
ramified Galois extension $L/K$ such that
\begin{enumerate}[(a)]
\item $\Gal(L/K)\cong G$.
\item The upper ramification breaks of $L/K$ are
$u_1,\ldots,u_n$.
\item The ramification subgroups of $\Gal(L/K)\cong G$
are the groups $G_{(i)}$ in the filtration of $G$.
\end{enumerate}
\end{cor}

\begin{proof}
If $\Sigma_G=\{1,2,\ldots,n\}$ set $M=1$.  Otherwise,
we use the notation of Corollary~\ref{ramfiltcor} to
define
\[M=\max\{p^{i-2}l_i:1\le i\le n,\;i\not\in\Sigma_G\}.\]
Let $u_1,\dots,u_n$ be positive integers such that
$p\nmid u_i$ for $1\le i\le n$ and $u_i>Mu_{i-1}$ for
$2\le i\le n$.  Then $u_i>u_{i-1}$, and for
$i\not\in\Sigma_G$ we get
$b_i\ge u_i>p^{i-2}l_iu_{i-1}$.  Therefore by
Corollary~\ref{ramfiltcor} there is an extension $L/K$
with the specified properties.
\end{proof}

     It would be interesting to know whether
Corollary~\ref{Mbreak} holds with $M=p$.

\section{Scaffolds, Galois module structure, and
Hopf orders}

Let $(G,\{G_{(i)}\})$ be a $p$-filtered group and let
$K$ be a local field of characteristic $p$ with perfect
residue field.  A Galois scaffold
$(\{\Psi_i\},\{\lambda_t\})$ for a $G$-extension $K_n/K$
consists of $\Psi_i\in K[G]$ for $1\le i\le n$ and
$\lambda_t\in K_n$ for all $t\in\Z$.  These are chosen
so that $v_{K_n}(\lambda_t)=t$ and $\Psi_i(\lambda_t)$
can be computed up to a certain ``precision'' $\cc\ge1$.
Note that if a Galois scaffold
$(\{\Psi_i\},\{\lambda_t\})$ for $K_n/K$ has precision
$\cc$, and $1\le\cc'\le\cc$, then it is also correct to
say that $(\{\Psi_i\},\{\lambda_t\})$ has precision
$\cc'$.  The existence of a Galois scaffold for $K_n/K$
facilitates the computation of the Galois module
structure of $\OO_{K_n}$ and its ideals.  For precise
definitions and some basic properties of Galois
scaffolds see \cite{bce}.

     In this section we show how the hypotheses of
Theorem~\ref{ramfilt} can be strengthened to guarantee
that the $G$-extension $K_n/K$ has a Galois scaffold.
This leads to sufficient conditions for $\OO_{K_n}$ be
be free over its associated order (Corollary~\ref{GMS}),
and sufficient conditions for the associated order of
$\OO_{K_n}$ to be a Hopf order
(Corollary~\ref{Hopforder}).

\begin{theorem} \label{scafcond}
Let $(G,\{G_{(i)}\})$ be a $p$-filtered group of order
$p^n$ and let $D_1,\dots,D_n$ be the polynomials
associated to $(G,\{G_{(i)}\})$ by Proposition~\ref{Di}.
Let $K$ be a local field of characteristic $p$ with
perfect residue field and let $a\in K^{\times}$ with
$p\nmid v_K(a)$.  For $1\le i\le n$ let
$\omega_i\in K^{\times}$ and set
$a_i=a\omega_i^{p^{n-1}}$.  Set $u_i=-v_K(a_i)$ and
assume that $0<u_1<\cdots<u_n$.  As in
Theorem~\ref{main} we define $K_0,K_1,\ldots,K_n$
recursively by $K_0=K$ and $K_i=K_{i-1}(\alpha_i)$ for
$1\le i\le n$, where $\alpha_i$ satisfies
$\alpha_i^p-\alpha_i=d_i+a_i$ with
$d_i=D_i(\alpha_1,\ldots,\alpha_{i-1})$.  Define
$b_1<b_2<\cdots<b_n$ recursively by $b_1=u_1$ and
$b_{i+1}-b_i=p^i(u_{i+1}-u_i)$ for $1\le i\le n-1$.  If
\begin{align} \label{weak1}
b_i&>-p^{n-1}v_K(d_i)-p^{n-i}b_{i-1}+p^{n-1}u_{i-1}, \\
b_i&>p^{n-1}u_{i-1} \label{weak2}
\end{align}
for all $2\le i\le n$ with $i\not\in\Sigma_G$ then the
extensions $K_0\subset K_1\subset\cdots\subset K_n$
satisfy conclusions (a)--(c) of Theorem~\ref{ramfilt},
plus the additional condition
\begin{enumerate}[(a)]
\setcounter{enumi}{3}
\item $K_n/K$ admits a Galois scaffold with precision
\[\cc=\min\{p^{n-1}v_K(d_i)+p^{n-i}b_{i-1}+b_i
-p^{n-1}u_{i-1},\:b_i-p^{n-1}u_{i-1}:
2\le i\le n,\:i\not\in\Sigma_G\}.\]
Furthermore, we have $\Psi_i\in K[G_{(n-i)}]$ for
$1\le i\le n$.
\end{enumerate}
\end{theorem}

\begin{remark} \label{gaps}
The precision $\cc$ given in the theorem is equal to the
minimum of the gaps in the inequalities (\ref{weak1})
and (\ref{weak2}).
\end{remark}

\begin{remark}
If $\Sigma_G=\{1,2,\dots,n\}$ then $G$ is an elementary
abelian $p$-group and our scaffold has infinite
precision (cf.\ the characteristic-$p$ case of
Theorem~3.5 of \cite{large}).
\end{remark}

\begin{proof}[Proof of Theorem \ref{scafcond}]
It follows from (\ref{weak1}) and Lemma~\ref{known}(b)
that for $i\not\in\Sigma_G$ we have
\begin{equation} \label{bibound}
b_i>-p^{n-1}v_K(d_i)-p^{n-i}b_{i-1}+p^{n-1}u_{i-1}
\ge-p^{n-1}v_K(d_i).
\end{equation}
Hence the extensions
$K_0\subset K_1\subset\cdots\subset K_n$ satisfy the
conclusions of Theorem~\ref{ramfilt}.  To prove (d) we
use \cite{large}, which gives a systematic method for
constructing Galois scaffolds.  By our assumptions on
$a_i$ we have $p\nmid u_1$ and
$u_i\equiv u_j\pmod{p^{n-1}}$ for $1\leq i,j\leq n$.
Thus $p\nmid b_1$ and $b_i\equiv b_j\pmod{p^n}$,
so Assumptions~2.2 and 2.6 of \cite{large} are
satisfied.  To apply Theorem~2.10 of \cite{large} we
must choose $\sigma_i\in\Gal(K_n/K_{i-1})$ as described
in Choice~2.1 of \cite{large}, and $\X_j\in K_j$ as
described in Choice~2.3 of \cite{large}.

     As in \cite{cyclic}, we begin by constructing
$\Y_j\in K_j$ such that
$v_{K_j}(\Y_j)\equiv-b_j\pmod{p^j}$.  We then obtain
$\X_j$ satisfying Choice~2.3 of \cite{large} by
multiplying $\Y_j$ by an appropriate element of
$K^{\times}$.  Set
\[\vomega=\begin{bmatrix}
\omega_1\\\omega_2\\\vdots\\\omega_j
\end{bmatrix}\in K^j \hspace{1cm}
\valpha=\begin{bmatrix}
\alpha_1\\\alpha_2\\\vdots\\\alpha_j
\end{bmatrix}\in(K^{sep})^j.\]
Let $\phi:(K^{sep})^j\ra(K^{sep})^j$ be the map induced
by the $p$-Frobenius on $K^{sep}$ and set
\begin{align} \nonumber
\Y_j&=\begin{vmatrix}
  \alpha_1&\omega_1^{p^{n-j}}&\cdots &\omega_1^{p^{n-2}}\\
  \alpha_2&\omega_2^{p^{n-j}}&\cdots &\omega_2^{p^{n-2}}\\
  \vdots & \vdots&&\vdots \\
  \alpha_j&\omega_j^{p^{n-j}}&\cdots &\omega_j^{p^{n-2}}\\
\end{vmatrix} \\[2mm]
&=\det[\valpha,\phi^{n-j}(\vomega),
\phi^{n-j+1}(\vomega),\ldots,\phi^{n-2}(\vomega)].
\label{Yjdet}
\end{align}
For $1\le i\le j$ we have $\alpha_i\in K_j$ and
$\omega_i\in K$.  Therefore $\Y_j\in K_j$.  

     For $1\le i\le j$ set $m_i=-v_K(\omega_i)$.  As in
the proof of Proposition~1 of \cite{cyclic}, we expand
in cofactors along the first column to get
\begin{equation} \label{defn-Y}
\Y_j=t_{1j}\alpha_1+t_{2j}\alpha_2+\cdots+t_{jj}\alpha_j,
\end{equation}
with $t_{ij}\in K$.
Since $m_1<\dots<m_j$, the $t_{ij}$ satisfy
\begin{align} \nonumber
v_K(t_{ij})&=v_K\left(\omega_1^{p^{n-j}}
\omega_2^{p^{n-j+1}}\dots\omega_{i-1}^{p^{n-j+i-2}}
\omega_{i+1}^{p^{n-j+i-1}}\dots\omega_j^{p^{n-2}}
\right), \\
&=-p^{n-j}(m_1+pm_2+\cdots+p^{i-2}m_{i-1}+p^{i-1}m_{i+1}+
\cdots+p^{j-2}m_j). \label{vKtik}
\end{align}
It follows that for $2\le i\le j$ we have
\begin{align*}
v_K(t_{ij})-v_K(t_{i-1,j})
&=-p^{n-j}(p^{i-2}m_{i-1}-p^{i-2}m_i)\\
&=p^{i-j-1}(p^{n-1}m_i-p^{n-1}m_{i-1})\\
&=p^{i-j-1}(u_i-u_{i-1})\\
&=p^{-j}(b_i-b_{i-1}).
\end{align*}
Hence $v_K(t_{ij})-v_K(t_{hj})$ is a telescoping sum for
$1\leq h\leq i\leq j$.  Therefore we get
\begin{align} \nonumber
v_K(t_{ij})-v_K(t_{hj})&=p^{-j}(b_i-b_h) \\
v_K(t_{ij}^{p^j})-v_K(t_{hj}^{p^j})&=b_i-b_h.
\label{v-t}
\end{align}

     We claim that for $0\le i\le j-1$ we have
\begin{equation} \label{phiiY}
\phi^i(\Y_j)=\det[\valpha+\d+\cdots+\phi^{i-1}(\d),
\phi^{n-j+i}(\vomega),\ldots,\phi^{n-2+i}(\vomega)].
\end{equation}
The case $i=0$ is given by (\ref{Yjdet}).  Let
$0\le i\le j-2$ and assume that (\ref{phiiY}) holds for
$i$.  Then
\begin{align*}
\phi^{i+1}(\Y_j)
&=\phi(\det[\valpha+\d+\cdots+\phi^{i-1}(\d),
\phi^{n-j+i}(\vomega),\ldots,\phi^{n-2+i}(\vomega)]) \\
&=\det[\phi(\valpha)+\phi(\d)+\cdots+\phi^i(\d),
\phi^{n-j+i+1}(\vomega),\ldots,\phi^{n-1+i}(\vomega)] \\
&=\det[\valpha+a\phi^{n-1}(\vomega)+\d+\phi(\d)+
\cdots+\phi^i(\d),
\phi^{n-j+i+1}(\vomega),\ldots,\phi^{n-1+i}(\vomega)].
\end{align*}
Since $n-j+i+1\le n-1\le n-1+i$ it follows that
\[\phi^{i+1}(\Y_j)
=\det[\valpha+\d+\phi(\d)+\cdots+\phi^i(\d),
\phi^{n-j+i+1}(\vomega),\ldots,\phi^{n-1+i}(\vomega)].\]
Hence (\ref{phiiY}) holds with $i$ replaced by $i+1$.

     It follows by induction that (\ref{phiiY}) holds
for $i=j-1$.  Therefore we have
\begin{align}
\phi^j(\Y_j)&=\phi(\det[\valpha+\d+\cdots+\phi^{j-2}(\d),
\phi^{n-1}(\vomega),\ldots,\phi^{n+j-3}(\vomega)])
\nonumber \\
&=\det[\phi(\valpha)+\phi(\d)+\cdots+\phi^{j-1}(\d),
\phi^n(\vomega),\ldots,\phi^{n+j-2}(\vomega)] \nonumber \\
&=\det[\valpha+a\phi^{n-1}(\vomega)+\d+\phi(\d)+
\cdots+\phi^{j-1}(\d),
\phi^n(\vomega),\ldots,\phi^{n+j-2}(\vomega)].
\label{phinY}
\end{align}
The $(i,1)$ cofactor of (\ref{phinY}) is $t_{ij}^{p^j}$,
where $t_{ij}$ is the $(i,1)$ cofactor of (\ref{Yjdet}).
Since $d_1=0$ this gives
\begin{equation} \label{Ypn}
\Y_j^{p^j}=t_{1j}^{p^j}(\alpha_1+a\omega_1^{p^{n-1}})
+\sum_{i=2}^jt_{ij}^{p^j}\left(\alpha_i
+a\omega_i^{p^{n-1}}+\sum_{h=0}^{j-1}d_i^{p^h}\right).
\end{equation}
Using (\ref{vKtik}) we get
\begin{align*}
v_K(t_{1j}^{p^j}a\omega_1^{p^{n-1}})
&=v_K(a\omega_1^{p^{n-1}})+p^jv_K(t_{1j}) \\
&=-b_1-p^nm_2-p^{n+1}m_3-\cdots-p^{n+j-2}m_j.
\end{align*}
We claim that $v_K(\Y_j^{p^j})
=v_K(t_{1j}^{p^j}a\omega_1^{p^{n-1}})$.
To prove this it suffices to show that the other terms
in (\ref{Ypn}) all have $K$-valuation greater than
$v_K(t_{1j}^{p^j}a\omega_1^{p^{n-1}})$.

     Since $v_K(\alpha_i)
=p^{-1}v_K(a\omega_i^{p^{n-1}})<0$ we have
$v_K(\alpha_i)>v_K(a\omega_i^{p^{n-1}})$ for
$1\leq i\leq j$.  Therefore that it suffices to prove
that
\begin{alignat}{2} \label{lt1}
v_K(t_{1j}^{p^j}a\omega_1^{p^{n-1}})
&<v_K(t_{ij}^{p^j}a\omega_i^{p^{n-1}})\quad
&&(2\leq i\leq j) \\
v_K(t_{1j}^{p^j}a\omega_1^{p^{n-1}})
&<v_K(t_{ij}^{p^j}d_i^{p^h})
&&(2\leq i\leq j,\;0\le h\le j-1). \label{lt2}
\end{alignat}
We first observe that (\ref{lt1}) follows from
(\ref{v-t}):
\begin{align*}
v_K(t_{ij}^{p^j}a\omega_i^{p^{n-1}})-v_K(t_{1j}^{p^j}a\omega_1^{p^{n-1}})
&=(v_K(t_{ij}^{p^j})-v_K(t_{1j}^{p^j}))+(v_K(a\omega_i^{p^{n-1}})-v_K(a\omega_1^{p^{n-1}}))\\
  &=(b_i-b_1)+(-u_i+u_1)\\
  &=b_i-u_i>0.
\end{align*}
We now prove (\ref{lt2}).  By (\ref{bibound}) we have
$b_i>-p^{n-1}v_K(d_i)$.  Since $h\leq n-1$ it follows
that $b_i>-p^hv_K(d_i)$.  Hence by \eqref{v-t} we get
\begin{align*}
v_K(t_{ij}^{p^j}d_i^{p^h})
-v_K(t_{1j}^{p^j}a\omega_1^{p^{n-1}})
&=b_i-b_1+p^hv_K(d_i)+u_1 \\
&=b_i+p^hv_K(d_i)>0.
\end{align*}
This proves (\ref{lt2}), so we have
\[v_K(\Y_j^{p^j})=v_K(t_{1j}^{p^j}a\omega_1^{p^{n-1}})
=v_K(t_{1j}^{p^j})-b_1.\]
Using (\ref{v-t}) we get
\[v_{K_j}(\Y_j)=v_K(t_{jj}^{p^j})-b_j
=v_{K_j}(t_{jj})-b_j.\]
We have $t_{jj}\not=0$ by (\ref{vKtik}), so we may
define $\X_j=t_{jj}^{-1}\Y_j$.  Then
$v_{K_j}(\X_j)=-b_j$, and since $t_{jj}\in K$ we also
have $v_{K_j}(\Y_j)\equiv-b_j\pmod{p^j}$.  Since
$p\nmid b_j$ it follows that $K_j=K(\X_j)=K(\Y_j)$.

     Now that we have constructed $\X_1,\dots,\X_n$,
we need to pick $\sigma_i\in\Gal(K_n/K_{i-1})$ for
$1\le i\le n$ which satisfy the conditions of
Choices~2.1 and 2.3 of \cite{large}.  Thus we need to
choose $\sigma_i\in\Gal(K_n/K_{i-1})$ such that
$\sigma_i|_{K_i}$ is a generator for
$\Gal(K_i/K_{i-1})\cong C_p$ and
$v_{K_i}((\sigma_i-1)\X_i-1)>0$.  To satisfy these
conditions it is enough to choose
$\sigma_i\in\Gal(K_n/K_{i-1})$ such that
$(\sigma_i-1)\alpha_i=1$.  We will be imposing
additional conditions on $\sigma_i$, namely that
$\sigma_i(\alpha_h)=\alpha_h$ for certain $h$ in the
range $i<h\le n$.  The purpose of these extra conditions
is to maximize the precision of the scaffold provided by
\cite{large} (see (\ref{c1new}) below).

     Recall from the proof of Theorem~\ref{ramfilt} that
$K_h=K_{h-1}(\alpha_h)$, where
$\alpha_h=\alpha_h'+\alpha_h''$, $\alpha_h'$ is a root
of $Y^p-Y-d_h$, and $\alpha_h''$ is a root of
$Y^p-Y-a_h$.  For $h\in\Sigma_G$ we have $d_h=0$, so we
may choose $\alpha_h'=0$, and hence
$\alpha_h=\alpha_h''$.  For $1\le i\le n$ set
\[A_i=\{\alpha_h:i<h\le n,\;h\in\Sigma_G\}.\]
Then $K_i(A_i)/K_{i-1}$ is an elementary abelian
$p$-extension of rank $|A_i|+1$ (see
Figure~\ref{fields}(a)).  Therefore there is
$\rho_i\in\Gal(K_i(A_i)/K_{i-1})$ such that
$(\rho_i-1)\alpha_i=1$ and $\rho_i(\alpha_h)=\alpha_h$
for all $\alpha_h\in A_i$.  Since
$K_i(A_i)\subset K_n$ there is
$\sigma_i\in\Gal(K_n/K_{i-1})$ such that
$\sigma_i|_{K_i(A_i)}=\rho_i$.  Then
$(\sigma_i-1)\alpha_i=1$ and
$\sigma_i(\alpha_h)=\alpha_h$ for all $h\in\Sigma_G$
such that $h\not=i$.  It follows that $\sigma_i|_{K_i}$
generates $\Gal(K_i/K_{i-1})$, so $\sigma_i$ satisfies
the conditions of Choice~2.1 of \cite{large}.  Since
$\sigma_i(\alpha_h)=\alpha_h$ for $1\le h<i$, by
\eqref{defn-Y} we get $(\sigma_i-1)\Y_i=t_{ii}$ and
$(\sigma_i-1)\X_i=1$.  Therefore $\sigma_i$ and $\X_i$
satisfy the conditions of Choice~2.3 of \cite{large}.

\begin{figure}
\caption{Field diagrams for Theorem~\ref{scafcond}}
\label{fields}
%\iffalse
\begin{center}

\begin{tikzpicture}[node distance = 2cm, auto]
\node (Ki1) {$K_{i-1}$};
\node (Ki) [above of=Ki1, left of=Ki1] {$K_i$};
\node (Ki1A) [above of=Ki1, right of=Ki1]
{$K_{i-1}(A_i)$};
\node (KiA) [above of=Ki, right of=Ki] {$K_i(A_i)$};
\node (Kn) [above of=KiA] {$K_n$};
\draw[-] (Ki) to (KiA);
\draw[-] (Ki1A) to (KiA);
\draw[-] (KiA) to (Kn);
\draw[-] (Ki1) to node [swap] {$C_p^{|A_i|}$} (Ki1A);
\draw[-] (Ki1) to node {$C_p$} (Ki);
\node [below of=Ki1, node distance=1cm] {(a)};
\end{tikzpicture}
\hspace{1cm}
\begin{tikzpicture}[node distance = 2cm, auto]
\node (Ki1) {$K_{i-1}$};
\node (Kh1) [above of=Ki1, left of=Ki1] {$K_{h-1}$};
\node (Ki1a'') [above of=Ki1, right of=Ki1]
{$K_{i-1}(\alpha_h'')$};
\node (Kh1a'') [above of=Kh1, right of=Kh1]
{$K_{h-1}(\alpha_h'')$};
\node (Kh) [above of=Kh1, left of=Kh1] {$K_h$};
\node (Kh1a') [above of=Kh1] {$K_{h-1}(\alpha_h')$};
\node (Kha'') [above of=Kh1a'] {$K_h(\alpha_h'')$};
\draw[-] (Ki1) to (Kh1);
\draw[-] (Ki1) to (Ki1a'');
\draw[-] (Kh1) to node [swap] {$C_p$} (Kh1a'');
\draw[-] (Ki1a'') to (Kh1a'');
\draw[-] (Kh1a'') to (Kha'');
\draw[-] (Kh1) to (Kh1a');
\draw[-] (Kh1) to node {$C_p$} (Kh);
\draw[-] (Kh) to (Kha'');
\draw[-] (Kh1a') to (Kha'');
\node [below of=Ki1, node distance=1cm] {(b)};
\end{tikzpicture}

\end{center}
%\fi
\end{figure}
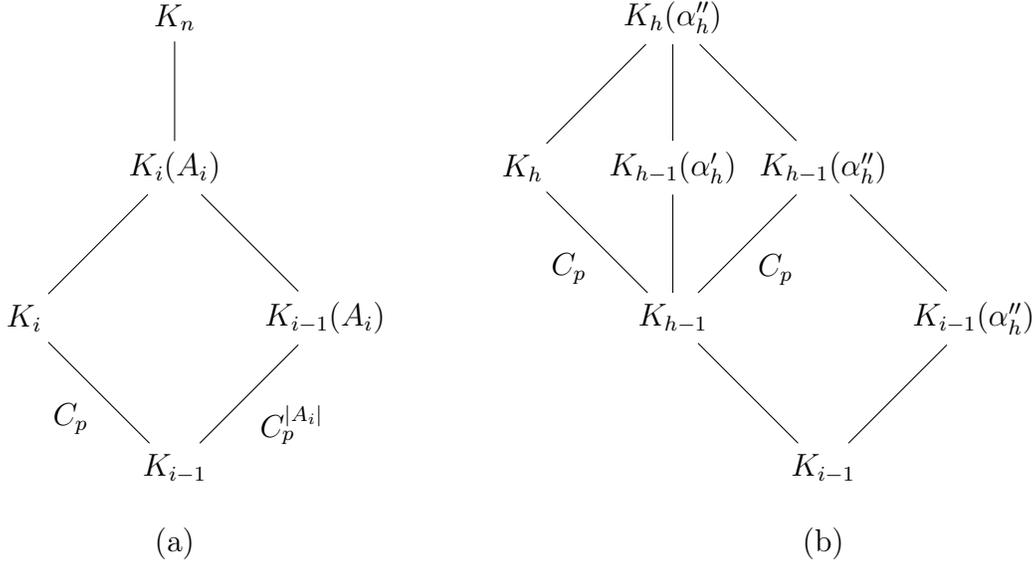

     In order to apply \cite{large} to get a Galois
scaffold for $K_n/K$ we need to look more closely at
the action of $K[\Gal(K_n/K)]$ on $K_n$.
Let $1\le i\le j\le n$.  By
Theorem~\ref{ramfilt}(b) the upper ramification breaks
of $K_j/K$ are $u_1,\ldots,u_j$.  Therefore the lower
ramification breaks of $K_j/K$ are $b_1,\ldots,b_j$.  In
particular, the lower ramification break
$i(\sigma_i|_{K_j})$ of $K_j/K$ associated to
$\sigma_i|_{K_j}$ is $b_i$.  Since
$p\nmid v_{K_j}(\X_j)$ this implies
\begin{equation} \label{sigiXj}
v_{K_j}((\sigma_i-1)\X_j)=v_{K_j}(\X_j)+b_i=b_i-b_j.
\end{equation}
By \eqref{defn-Y} we have
\begin{equation} \label{mueps}
(\sigma_i-1)\X_j=\mu_{ij}+\epsilon_{ij},
\end{equation}
with $\mu_{ij}=t_{ij}/t_{jj}$ and
\[\epsilon_{ij}=\frac{t_{i+1,j}}{t_{jj}}(\sigma_i-1)
\alpha_{i+1}+\cdots
+\frac{t_{j-1,j}}{t_{jj}}(\sigma_i-1)\alpha_{j-1}
+(\sigma_i-1)\alpha_j.\]
Then $\mu_{ij}\in K$ and $\epsilon_{ij}\in K_j$.
Furthermore, $\mu_{ii}=1$, $\epsilon_{ii}=0$, and for
$1\le i<j\le n$ we have
\begin{equation} \label{diff}
v_{K_n}(\epsilon_{ij})-v_{K_n}(\mu_{ij})\geq
\min\{v_{K_n}(t_{hj}(\sigma_i-1)\alpha_h):i<h\leq j\}
-v_{K_n}(t_{ij}).
\end{equation}
We view $\mu_{ij}$ as the ``main term'' and
$\epsilon_{ij}$ as the ``error term'' in the
decomposition (\ref{mueps}) of $(\sigma_i-1)\X_j$.

     Motivated by Assumption~2.9 of \cite{large} we
define
\[\cc_0:=\min\{v_{K_n}(\epsilon_{ij})-v_{K_n}(\mu_{ij})
-p^{n-1}u_i+p^{n-j}b_i:1\leq i<j\leq n\}.\]
Assume $\cc_0\ge1$.  Since $i<j$ it follows from
Lemma~\ref{known}(b) that $-p^{n-1}u_i+p^{n-j}b_i\le0$.
Hence the right side of (\ref{diff}) is positive for all
$1\le i<j\le n$.  Thus
$v_{K_j}(\mu_{ij})<v_{K_j}(\epsilon_{ij})$, so by
(\ref{sigiXj}) we get
\[b_i-b_j=v_{K_j}((\sigma_i-1)\X_j)=v_{K_j}(\mu_{ij}).\]
Therefore (\ref{mueps}) satisfies the conditions of
equation (5) of \cite{large}.  We can now apply
Theorem~2.10 of \cite{large} which says that $K_n/K$
admits a Galois scaffold $(\{\Psi_i\},\{\lambda_t\})$
with precision $\cc_0$.  The operators $\Psi_i$ are
defined recursively in Definition~2.7 of \cite{large}
using $\mu_{ij}\in K$ and $\sigma_i\in G_{(n-i)}$.
Therefore we have $\Psi_i\in K[G_{(n-i)}]$.

     It remains to show that $\cc_0\ge\cc$, where $\cc$
is the precision given in the statement of the theorem.
Using \eqref{v-t} we get
$v_{K_n}(t_{hj})-v_{K_n}(t_{ij})=p^{n-j}(b_h-b_i)$.
Therefore we can rewrite (\ref{diff}) as
\begin{equation} \label{rewrite}
v_{K_n}(\epsilon_{ij})-v_{K_n}(\mu_{ij})
\geq\min\{v_{K_n}((\sigma_i-1)\alpha_h)+p^{n-j}(b_h-b_i):
i<h\leq j\}.
\end{equation}
Set
\begin{equation} \label{c1}
\cc_1=\min\{v_{K_n}((\sigma_i-1)\alpha_h)+p^{n-j}b_h
-p^{n-1}u_i:1\leq i<h\leq j\leq n\}.
\end{equation}
Then by (\ref{rewrite}) we get $\cc_0\ge\cc_1$.  Hence
if $\cc_1\geq 1$ then $K_n/K$ has a Galois scaffold with
precision $\cc_1$.  For fixed $1\le i<h\le n$ the
expression in (\ref{c1}) is minimized by taking $j=n$.
Hence
\[\cc_1=\min\{v_{K_n}((\sigma_i-1)\alpha_h)+b_h
-p^{n-1}u_i:1\leq i<h\leq n\}.\]
Recall that $\sigma_i$ was chosen so that
$(\sigma_i-1)\alpha_h=0$ for all $h\in\Sigma_G$.
Therefore we have
\begin{equation} \label{c1new}
\cc_1=\min\{v_{K_n}((\sigma_i-1)\alpha_h)+b_h
-p^{n-1}u_i:1\leq i<h\leq n,\:h\not\in\Sigma_G\}.
\end{equation}

     Let $1\leq i<h\le n$ with $h\not\in\Sigma_G$.  In
the proof of Theorem~\ref{ramfilt} we saw that
$K_h(\alpha_h'')=K_{h-1}(\alpha_h',\alpha_h'')$ is a
$(C_p\times C_p)$-extension of $K_{h-1}$ (see
Figure~\ref{fields}(b)).  Therefore
$\alpha_h''\not\in K_h$.  Let $\tau_{ih}$ be the
(uniquely determined) element of
$\Gal(K_h(\alpha_h'')/K_{i-1}(\alpha_h''))$ such that
$\tau_{ih}|_{K_h}=\sigma_i|_{K_h}$.  Since
$\tau_{ih}(\alpha_h'')=\alpha_h''$ we get
\[(\sigma_i-1)(\alpha_h)=(\tau_{ih}-1)(\alpha_h)
=(\tau_{ih}-1)(\alpha_h').\]
Since $\alpha_h'$ is a root of $Y^p-Y-d_h$, it follows
that $(\sigma_i-1)(\alpha_h)=(\tau_{ih}-1)(\alpha_h')$
is a root of $Y^p-Y-(\sigma_i-1)d_h$.  We have
\begin{align*}
v_{K_{h-1}}((\sigma_i-1)d_h)&\ge v_{K_{h-1}}(d_h)+b_i.
\end{align*}
It follows that
\begin{align*}
v_{K_h}((\sigma_i-1)\alpha_h)
&=v_{K_h}((\tau_{ih}-1)\alpha_h') \\
&\ge\min\{v_{K_{h-1}}((\sigma_i-1)d_h),0\} \\
&\ge\min\{v_{K_{h-1}}(d_h)+b_i,0\} \\
&=\min\{p^{h-1}v_K(d_h)+b_i,0\},
\end{align*}
and hence that
\begin{align} \label{vKnsig}
v_{K_n}((\sigma_i-1)\alpha_h)
&\ge\min\{p^{n-1}v_K(d_h)+p^{n-h}b_i,0\}.
\end{align}

     Set
\[\cc_2=\min\{p^{n-1}v_K(d_h)+p^{n-h}b_i+b_h
-p^{n-1}u_i,\:b_h-p^{n-1}u_i:
1\leq i<h\leq n,\:h\not\in\Sigma_G\}.\]
Then $\cc_1\ge\cc_2$ by (\ref{c1new}) and
(\ref{vKnsig}).  Hence if $\cc_2\ge1$ then $K_n/K$ has a
Galois scaffold with precision $\cc_2$.  Fix
$2\le h\le n$.  Using Lemma~\ref{known}(a) (with
$j=h-1$) we see that the two expressions in the formula
for $\cc_2$ are minimized by taking $i=h-1$.  Therefore
\[\cc_2=\min\{p^{n-1}v_K(d_h)+p^{n-h}b_{h-1}+b_h
-p^{n-1}u_{h-1},\:b_h-p^{n-1}u_{h-1}:2\leq h\leq n,\:
h\not\in\Sigma_G\}.\]
Thus $\cc_2$ is equal to the precision $\cc$ given in
the statement of the theorem.  We have $\cc\ge1$ by
assumptions (\ref{weak1}) and (\ref{weak2}).  It now
follows from Theorem~2.10 of \cite{large} that $K_n/K$
has a Galois scaffold with precision $\cc$.
\end{proof}

\begin{cor} \label{scafcondcor}
Let $(G,\{G_{(i)}\})$ be a $p$-filtered group of order
$p^n$, with $n\ge2$.  Let $D_1,\dots,D_n$ be the
polynomials associated to $(G,\{G_{(i)}\})$ by
Proposition~\ref{Di}, and for $i\not\in\Sigma_G$ let
$l_i$ be the total degree of $D_i$.  Choose positive
integers $u_1<\cdots<u_n$ with $p\nmid u_1$ and
$u_i\equiv u_1\pmod{p^{n-1}}$ for $2\le i\le n$.  Define
$b_1<b_2<\cdots<b_n$ recursively by $b_1=u_1$ and
$b_{i+1}-b_i=p^i(u_{i+1}-u_i)$ for $1\le i\le n-1$.
Assume that $u_1,\dots,u_n$ have been chosen so that
\begin{equation} \label{stronger}
b_i>p^{n-2}l_iu_{i-1}-p^{n-i}b_{i-1}+p^{n-1}u_{i-1}
\end{equation}
for all $2\le i\le n$ with $i\not\in\Sigma_G$.  Then
there exists a tower of extensions
$K=K_0\subset K_1\subset\cdots\subset K_n$ satisfying
(a)--(c) of Theorem~\ref{ramfilt}, plus the additional
condition
\begin{enumerate}[(a)]
\setcounter{enumi}{3}
\item $K_n/K$ admits a Galois scaffold with precision
\begin{equation} \label{cc}
\cc'=\min\{p^{n-i}b_{i-1}-p^{n-2}l_iu_{i-1}+b_i
-p^{n-1}u_{i-1}:2\le i\le n,\:i\not\in\Sigma_G\}.
\end{equation}
\end{enumerate}
\end{cor}

\begin{proof}
It follows from the assumptions on $u_1,\dots,u_n$ that
there are $a,\omega_i\in K^{\times}$ such that
$v_K(a\omega_i^{p^{n-1}})=-u_i$ for $1\le i\le n$.
Since $v_K(a)\equiv-u_1\pmod{p^{n-1}}$ we have
${p\nmid v_K(a)}$.  It follows from (\ref{stronger}) and
Lemma~\ref{known}(b) that $b_i>p^{n-2}l_iu_{i-1}$ for
all $2\le i\le n$ with $i\not\in\Sigma_G$.  Hence the
proof of Corollary~\ref{ramfiltcor} shows that
$p^{n-2}l_iu_{i-1}\ge-p^{n-1}v_K(d_i)$ for all
$2\le i\le n$ such that $i\not\in\Sigma_G$.  Therefore
(\ref{weak1}) follows from (\ref{stronger}).  Using
Lemma~\ref{known}(b) we get
$p^{i-2}l_iu_{i-1}\ge p^{i-2}u_{i-1}\ge b_{i-1}$.  Hence
(\ref{weak2}) also follows from (\ref{stronger}).  Thus
Theorem~\ref{scafcond} gives a tower of extensions
$K=K_0\subset K_1\subset\cdots\subset K_n$ satisfying
the conditions (a)--(d) given there.  The inequalities
above also imply
\begin{align*}
-p^{n-2}l_iu_{i-1}+p^{n-i}b_{i-1}+b_i-p^{n-1}u_{i-1}
&\le p^{n-1}v_K(d_i)+p^{n-i}b_{i-1}+b_i-p^{n-1}u_{i-1}
\\
-p^{n-2}l_iu_{i-1}+p^{n-i}b_{i-1}+b_i-p^{n-1}u_{i-1}
&\le b_i-p^{n-1}u_{i-1}.
\end{align*}
Therefore the scaffold given by
Theorem~\ref{scafcond}(d) has the precision $\cc'$
specified in (\ref{cc}).
\end{proof}

\begin{remark}
Suppose $G\cong C_{p^n}$ is cyclic.  Theorem~2 of
\cite{cyclic} gives a Galois scaffold with precision
\[\cc_0=\min\{b_i-p^nu_{i-1}:2\leq i\leq n\},\]
under the assumption that $b_i>p^nu_{i-1}$ for
$2\leq i\leq n$.  Since $G$ is cyclic we have
$\Sigma_G=\{1\}$.  Furthermore, by Lemma~4(a) of
\cite{cyclic} we get $v_K(d_i)\geq-pu_i$.
As in the proof of Corollary~\ref{scafcondcor} we can
apply Theorem~\ref{scafcond} to produce a Galois
scaffold with precision
\[\cc_1=\min\{p^{n-i}b_{i-1}-p^nu_i+b_i-p^{n-1}u_{i-1}:
2\leq i\leq n\},\]
under the assumption that
$b_i>p^nu_i-p^{n-i}b_{i-1}+p^{n-1}u_{i-1}$ for
$2\leq i\leq n$.  If $n\ge1$ then the precision $\cc_1$
is strictly less than the precision $\cc_0$ of
\cite{cyclic}.  Furthermore, Theorem~2 of \cite{cyclic}
allows more general choices of $a_i\in K$, namely
$a_i=a\omega_i^{p^{n-1}}+e_i$ for any $e_i\in K$ such
that $v_K(e_i)-v_K(a_i)$ satisfies the lower bound given
in assumption (3.3) of \cite{cyclic}.
\end{remark}

\begin{remark}
It follows from Corollary~\ref{scafcondcor} that by
choosing $u_1,\ldots,u_n$ which grow quickly enough we
can make $\cc$ arbitrarily large.
\end{remark}

     The scaffolds that we obtain from
Theorem~\ref{scafcond} can be used to get information
about Galois module structure.  Let $L/K$ be a Galois
extension with Galois group $G$.  Recall that the
associated order of $\OO_L$ in $K[G]$ is defined to be
\[\AA_0=\{\gamma\in K[G]:\gamma(\OO_L)\subset\OO_L\}.\]

\begin{cor}\label{GMS}
Let $(G,\{G_{(i)}\})$ be a $p$-filtered group of order
$p^n$ and let $K_n/K$ be a $G$-extension satisfying the
conditions of Theorem~\ref{scafcond}.  Let
$u_1<\dots<u_n$ be the upper ramification breaks of
$K_n/K$ and let $r(u_1)$ be the least nonnegative
residue of $u_1$ modulo $p^n$.  Assume that $r(u_1)\mid
p^m-1$ for some $1\le m\le n$ and that the precision
$\cc$ of the scaffold provided by Theorem~\ref{scafcond}
satisfies $\cc\ge r(u_1)$.  Then $\OO_{K_n}$ is free
over its associated order $\AA_0$.
\end{cor}

\begin{proof}
Since $u_i\equiv u_1\pmod{p^{n-1}}$ for $1\le i\le n$ we
have $b_i\equiv b_1\pmod{p^n}$.  It follows that
$r(b_n)=r(b_1)=r(u_1)$.  Theorem~\ref{scafcondcor}
shows that $K_n/K$ has a Galois scaffold with precision
$\cc\ge r(u_1)$.  Hence the corollary follows from
Theorem~4.8 of \cite{bce}.
\end{proof}

     Let $K$ be a local field with residue
characteristic $p$.  Let $G$ be a finite group and let
$H$ be an $\OO_K$-order in $K[G]$.  Say that $H$ is a
{\em Hopf order} if $H$ is a Hopf algebra over $\OO_K$
with respect to the operations inherited from the
$K$-Hopf algebra $K[G]$.  Say that the Hopf order
$H\subset K[G]$ is {\em realizable} if there is a
$G$-extension $L/K$ such that $H$ is equal to the
associated order $\AA_0$ of $\OO_L$ in $K[G]$.  A great
deal of eﬀort has gone into constructing and classifying
Hopf orders in $K[C_p^n]$ and $K[C_{p^n}]$; see
Chapter~12 of \cite{hopf} for a summary.  The only
method known for constructing Hopf orders in $K[G]$ for
an arbitrary $p$-group $G$ was given by Larson
\cite{lar}.  However, Larson's group-theoretic approach
does not give a method for finding Hopf orders which are
realizable.  Therefore it is interesting that in
the case where $\ch(K)=p$ the scaffolds from
Theorem~\ref{scafcond} can be used to construct
realizable Hopf orders in $K[G]$.  Since these Hopf
orders are constructed using the main result of
\cite{large}, they are ``truncated exponential Hopf
orders'' in the sense of \cite[\S12.9]{hopf}.  Thus one
consequence of the following corollary is that for all
$p$-groups $G$, truncated exponential Hopf orders exist
in $K[G]$.

\begin{cor}\label{Hopforder}
Let $(G,\{G_{(i)}\})$ be a $p$-filtered group of order
$p^n$ and let $K_n/K$ be a $G$-extension satisfying the
conditions of Theorem~\ref{scafcond}.  Let
$u_1<\dots<u_n$ be the upper ramification breaks of
$K_n/K$ and assume that $u_1\equiv-1\bmod p^n$.  Assume
further that the precision $\cc$ of the scaffold
provided by Theorem~\ref{scafcond} satisfies
$\cc\geq p^n-1$.  Then the associated order $\AA_0$ of
$\OO_{K_n}$ in $K[G]$ is a Hopf order.
\end{cor}

\begin{proof}
It follows from the preceding corollary that $\OO_{K_n}$
is free over $\AA_0$.  The action of $K[G]$ on $K_n$ is
the regular representation, which is indecomposable
since $\ch(K)=p$.  It follows that $\OO_{K_n}$ is
indecomposable as an $\OO_K[G]$-module.  Furthermore,
since $b_i\equiv-1\pmod{p^n}$ for $1\le i\le n$, the
different of $L/K$ is generated by an element of $K$.
Hence by Proposition~4.5.2 of \cite{bon2} we deduce that
$\AA_0$ is a Hopf order in $K[G]$.
\end{proof}

\begin{remark}
Let $K$ be a local field of characteristic 0 with
residue characteristic $p$ and let $G$ be a finite
noncyclic $p$-group.  In \cite{bd} Bondarko and Dievsky
gave necessary and sufficient conditions on a Hopf
order $H\subset K[G]$ for $H$ to be realizable by a
totally ramified $G$-extension $L/K$ whose different is
generated by an element of $K$.  However, it is
difficult to construct Hopf algebras that satisfy
their criterion.  Byott \cite{mono} showed that if $G$
is abelian and the $\OO_K$-dual $H^*$ of $H$ is a
monogenic $\OO_K$-algebra then $H$ is realizable.
\end{remark}

\section{Dihedral examples}

Let $G$ be the dihedral group of order 16.  Write
$G=\langle \sigma,\tau\rangle$ with $\sigma$ a rotation
of order 8 and $\tau$ a reflection.  We define a
2-filtration of $G$ by setting $G_{(0)}=G$,
$G_{(1)}=\langle\sigma^2,\tau\rangle$,
$G_{(2)}=\langle\sigma^2\rangle$,
$G_{(3)}=\langle\sigma^4\rangle$, and $G_{(4)}=\{1\}$.
Then $\Phi(G)=\langle\sigma^2\rangle=G_{(2)}$, so we
have $\Sigma_G=\{1,2\}$.  Let $K$ be a local field of
characteristic $p=2$.  We will use the methods we have
developed to give three examples of $G$-extensions
$K_4/K$ with specified properties.

     We first construct a generic $G$-extension of
rings using the results of Section~\ref{generic}.  Since
$\Sigma_G=\{1,2\}$ we have $D_1=D_2=0$.  Therefore
$X_1,X_2$ are elements of $S_2\cong\F_2[Y_1,Y_2]$
defined by $X_1=Y_1^2-Y_1$ and $X_2=Y_2^2-Y_2$.  To
determine $D_3$ we use the procedure outlined in the
paragraph following Proposition~\ref{tensor}.  Set
$\sigmab=\sigma G_{(2)}$, $\taub=\tau G_{(2)}$,
$\sigmat=\sigma G_{(3)}$, and $\taut=\tau G_{(3)}$.
Then $(\sigmab-1)Y_1=1$, $(\sigmab-1)Y_2=0$,
$(\taub-1)Y_1=0$, and $(\taub-1)Y_2=1$.  Let
$u:G/G_{(2)}\ra G/G_{(3)}$ be the section of the
projection $\pi:G/G_{(3)}\ra G/G_{(2)}$ whose image is
$\{\onet,\sigmat,\taut,\sigmat\taut\}$, and let $\chi$
be the unique isomorphism from $G_{(2)}/G_{(3)}$ to
$\F_2$.  Then the 2-cocycle
$c:(G/G_{(2)})\times(G/G_{(2)})\ra\F_2$ defined by
$c(g,h)=\chi(u(g)u(h)u(gh)^{-1})$ represents the class
in $H^2(G/G_{(2)},\F_2)$ which corresponds to the group
extension $\pi:G/G_{(3)}\ra G/G_{(2)}$.  We find that
the cochain $(s_g)_{g\in G/G_{(2)}}$ defined by
$s_{\oneb}=0$, $s_{\sigmab}=s_{\taub}=Y_1$, and
$s_{\sigmab\taub}=1$, satisfies
$c(g,h)=s_g+g(s_h)-s_{gh}$ for all $g,h\in G/G_{(2)}$.
We have
\begin{align*}
\wp(s_{\oneb})&=0=(\oneb-1)(X_1(Y_1+Y_2)) \\
\wp(s_{\sigmab})
&=X_1=(\sigmab-1)(X_1(Y_1+Y_2)) \\
\wp(s_{\taub})&=X_1=(\taub-1)(X_1(Y_1+Y_2)) \\
\wp(s_{\sigmab\taub})
&=0=(\sigmab\taub-1)(X_1(Y_1+Y_2)).
\end{align*}
Therefore we can take
\begin{align*}
D_3&=X_1(Y_1+Y_2)=(Y_1^2-Y_1)(Y_1+Y_2) \\
X_3&=Y_3^2-Y_3-(Y_1^2-Y_1)(Y_1+Y_2).
\end{align*}
A similar but more complicated computation based on the
formulas $(\sigmat-1)Y_1=1$, $(\sigmat-1)Y_2=0$,
$(\sigmat-1)Y_3=Y_1$, $(\taut-1)Y_1=0$,
$(\taut-1)Y_2=1$, and $(\taut-1)Y_3=Y_1$ gives
\[D_4=X_1^3Y_1+X_1^2X_2Y_2+X_1^2Y_1Y_2+ 
X_1(Y_1^3+Y_1Y_3+Y_2Y_3+Y_2)+X_1X_3(Y_1+Y_2)+
X_3(Y_3+Y_2).\]
We can represent $D_4$ as a polynomial in $Y_1,Y_2,Y_3$
by expressing $X_1,X_2,X_3$ in terms of $Y_1,Y_2,Y_3$
using the formulas given above.

     We now use the generic $G$-extension of rings
that we have constructed to get a family of
$G$-extensions of $K$.  Let $a_1,a_2,a_3,a_4\in K$ and
set $u_i=-v_K(a_i)$.  Assume that
$0<u_1<u_2<u_3<u_4$ and $u_1,u_2,u_3,u_4$ are odd.  Define $b_1,b_2,b_3,b_4$ by $b_1=u_1$ and
$b_{i+1}=b_i+2^i(u_{i+1}-u_i)$ for $1\le i\le3$.  Set
$K_4=K(\alpha_1,\alpha_2,\alpha_3,\alpha_4)$, where the
$\alpha_i$ are defined recursively by
$\alpha_i^2-\alpha_i=d_i+a_i$, with $d_1=d_2=0$,
$d_3=D_3(\alpha_1,\alpha_2)$, and
$d_4=D_4(\alpha_1,\alpha_2,\alpha_3)$.  Since
$u_1,u_2$ are distinct, positive, and odd,
$\{a_1+\wp(K),a_2+\wp(K)\}$ is linearly
independent over $\F_p$.  Therefore it follows from
Theorem~\ref{main} that $K_4/K$ is a $G$-extension.  By
putting additional conditions on $a_1,a_2,a_3,a_4$ we
will get examples of $G$-extensions which have various
interesting properties.

\begin{example}
To satisfy the hypotheses of Theorem~\ref{ramfilt} we
need to choose $a_i$ so that $b_i>p^{i-1}v_K(d_i)$ for
$i=3,4$.  We first choose $a_1,a_2$ so that $0<u_1<u_2$
are odd.  This gives $b_1=u_1$, $b_2=2u_2-u_1$, and
$v_K(d_3)=-u_1-\frac12u_2$.  We must choose $a_3$ so
that $u_3$ is odd and $b_3=4u_3-2u_2-u_1$ is greater
than $-4v_K(d_3)=4u_1+2u_2$.  This is equivalent to
$u_3>\frac54u_1+u_2$.  Under this assumption we have
\[v_K(d_4)\ge
\tst\min\{-u_1-\frac12u_2-u_3,-\frac32u_3\}\]
and $b_4=8u_4-4u_3-2u_2-u_1$.  Therefore it suffices to
choose $a_4$ so that $u_4$ satisfies
\[8u_4-4u_3-2u_2-u_1>\max\{8u_1+4u_2+8u_3,12u_3\}.\]
This is equivalent to
\[u_4>\tst\max\{\frac{9}{8}u_1+\frac{3}{4}u_2
+\frac{3}{2}u_3,\,\frac{1}{8}u_1+\frac{1}{4}u_2+2u_3\}.\]
If these conditions are satisfied then it follows from
Theorem~\ref{ramfilt} that $K_4/K$ is a $G$-extension
whose upper ramification breaks are $u_1,u_2,u_3,u_4$.
To get a specific example we let $\pi_K$ be a
uniformizer for $K$ and set $a_1=\pi_K^{-1}$,
$a_2=\pi_K^{-3}$, $a_3=\pi_K^{-5}$, $a_4=\pi_K^{-11}$.
This gives a $G$-extension $K_4/K$ with upper
ramification breaks $1,3,5,11$ and lower ramification
breaks $1,5,13,61$.
\end{example}

\begin{example}
In order to use Theorem~\ref{scafcond} to get a
$G$-extension $K_4/K$ with a Galois scaffold we write
$a_i=a\omega_i^8$ and
consider the possibilities for the ramification data of
$K_4/K$.  Choose $u_1=b_1=1$.  We need $u_2>u_1$ with
$u_2\equiv u_1\pmod8$, so we choose $u_2=9$.  It follows
that $b_2=1+2(9-1)=17$.  We need $u_3>u_2$ with
$u_3\equiv1\pmod8$ such that $b_3=17+4(u_3-9)$ satisfies
\begin{align*}
b_3&>\tst8\cdot\frac{11}{2}-2\cdot17+8\cdot9=82 \\
b_3&>\tst8\cdot9=72.
\end{align*}
We choose $u_3=33$, so $b_3=113$.  Finally, we need
$u_4>u_3$ with $u_4\equiv1\pmod8$ such that
$b_4=113+8(u_4-33)$ satisfies
\begin{align*}
b_4&>\tst8\cdot\max\{1+\frac12\cdot9+33,\frac32\cdot33\}
-113+8\cdot33=547 \\
b_4&>8\cdot33=264.
\end{align*}
We choose $u_4=89$, which gives $b_4=561$.  This
ramification data can be realized by taking
$a=\pi_K^{-1}$, $\omega_1=1$, $\omega_2=\pi_K^{-1}$,
$\omega_3=\pi_K^{-4}$, and $\omega_4=\pi_K^{-11}$.
According to Theorem~\ref{scafcond} and
Remark~\ref{gaps}, these choices give a $G$-extension
$K_4/K$ which has a Galois scaffold with precision
\[\cc=\min\{b_3-82,b_3-72,b_4-547,b_4-264\}=14.\]
It then follows from Corollary~\ref{GMS} that
$\OO_{K_4}$ is free over its associated order $\AA_0$.
\end{example}

\begin{example}
We wish to use Corollary~\ref{Hopforder} to produce a
$G$-extension $K_4/K$ such that the associated order
$\AA_0$ of $\OO_{K_4}$ in $K[G]$ is a Hopf order.  Once
again we set $a_i=a\omega_i^8$.  We need to determine
ramification data for $K_4/K$ that satisfies the
hypotheses of Corollary~\ref{Hopforder}.  The first
requirement is $u_1\equiv-1\pmod{16}$, so we choose
$u_1=b_1=15$.  We need $u_2>u_1$ with
$u_2\equiv-1\pmod8$.  We choose $u_2=23$ and hence
$b_2=31$.  To apply Corollary~\ref{Hopforder} we need to
construct an extension which has a scaffold with
precision $\cc\ge2^4-1=15$.  Therefore we wish to find
$u_3\equiv-1\pmod8$ such that $b_3=31+4(u_3-23)$ makes
the gaps in inequalities (\ref{weak1}) and (\ref{weak2})
greater than or equal to 15 (see Remark~\ref{gaps}).
Hence we require
\begin{align*}
b_3&\geq\tst8\cdot\frac{53}{2}-2\cdot31+8\cdot23+15=349, \\
b_3&\geq\tst8\cdot23+15=199.
\end{align*}
By choosing $u_3=103$ we get $b_3=351$, which satisfies
both inequalities.  Similarly, we need $u_4>u_3$ with
$u_4\equiv-1\pmod8$ such that $b_4=351+8(u_4-103)$
satisfies
\begin{align*}
b_4&\geq\tst8\cdot\max\{15+\frac12\cdot23+103,\frac32\cdot103\}
-351+8\cdot103+15=1724, \\
b_4&\geq8\cdot103+15=839.
\end{align*}
We choose $u_4=279$, which gives $b_4=1759$.  We get a
$G$-extension $K_4/K$ with this ramification data by
taking $a=\pi_K^{-15}$, $\omega_1=1$,
$\omega_2=\pi_K^{-1}$, $\omega_3=\pi_K^{-11}$, and
$\omega_4=\pi_K^{-33}$.  Using the definitions of
$\mu_{ij}$ in (\ref{mueps}) and $t_{ij}$ in
(\ref{defn-Y}) we get
\begin{align*}
\mu_{12}&=\frac{1}{\pi_K^4}, \\[1mm]
\mu_{13}&=\frac{1+\pi_K^{20}}{\pi_K^{42}(1+\pi_K^2)},
\\[1mm]
\mu_{14}&=\frac{1+\pi_K^{10}+\pi_K^{44}+\pi_K^{74}
+\pi_K^{76}+\pi_K^{96}}{\pi_K^{109}(1+\pi_K+\pi_K^{20}
+\pi_K^{23}+\pi_K^{31}+\pi_K^{33})}, \\[1mm]
\mu_{23}&=\frac{1+\pi_K^{22}}{\pi_K^{40}(1+\pi_K^{2})},
\\[1mm]
\mu_{24}&=\frac{1+\pi_K^{11}+\pi_K^{44}
+\pi_K^{99}}{\pi_K^{108}(1+\pi_K+\pi_K^{20}+\pi_K^{23}
+\pi_K^{31}+\pi_K^{33})}, \\[1mm]
\mu_{34}&=\frac{1+\pi_K+\pi_K^{64}+\pi_K^{67}
+\pi_K^{97}+\pi_K^{99}}{\pi_K^{88}(1+\pi_K+\pi_K^{20}
+\pi_K^{23}+\pi_K^{31}+\pi_K^{33})}.
\end{align*}
Definition~2.7 of \cite{large} computes $\Theta_i\in
K[G]$ recursively in terms of the ``truncated
exponential'' $X^{[Y]}=1+Y(X-1)$.  We find that
$\Theta_4=\sigma^4$, $\Theta_3
=\sigma^2\Theta_4^{[-\mu_{3,4}]}$, $\Theta_2
=\sigma\Theta_3^{[-\mu_{2,3}]}\Theta_4^{[-\mu_{2,4}]}$,
and $\Theta_1=\tau\Theta_2^{[-\mu_{1,2}]}
\Theta_3^{[-\mu_{1,3}]}\Theta_4^{[-\mu_{1,4}]}$.  For
$1\le i\le4$ set $M_i=(b_i+1)/p^i$.  It follows from
equation (34) of \cite{large} that the associated order
$\AA_0$ of $\OO_{K_4}$ in $K[G]$ is
\begin{align*}
\AA_0&=\OO_{K}\left[\frac{\Theta_4-1}{\pi_K^{M_4}},
\frac{\Theta_3-1}{\pi_K^{M_3}},
\frac{\Theta_2-1}{\pi_K^{M_2}},
\frac{\Theta_1-1}{\pi_K^{M_1}}\right] \\
&=\OO_{K}\left[ \frac{\Theta_4-1}{\pi_K^{110}},\frac{\Theta_3-1}{\pi_K^{44
}},\frac{\Theta_2-1}{\pi_K^8},\frac{\Theta_1-1}{\pi_K^8}\right].
\end{align*}
By Corollary~\ref{Hopforder} we see that $\AA_0$ is a
Hopf order in $K[G]$.
\end{example}

\end{document}